\documentclass{autart}

\usepackage{alltt}
\usepackage{graphics}
\usepackage{graphicx,epsfig}
\usepackage{subcaption}
\usepackage{epsfig}
\usepackage{psfrag}
\usepackage{makeidx}
\usepackage{tabularx}
\usepackage{multicol,multirow}
\usepackage{enumitem}

\usepackage{amscd,amsfonts,amsmath}
\usepackage{amssymb}
\usepackage{graphicx}
\usepackage{tikz}
\usetikzlibrary{decorations.pathreplacing,calc}
\usetikzlibrary{matrix,decorations.pathreplacing,calc}
                            
\usepackage{color}

\newcommand{\fin}{\hfill$\Box$}

\newcommand{\n}{\normalfont}

\newcommand{\G}{K}
\newcommand{\Lo}{\mathbb{L}}
\newcommand{\sLo}{\sigma\mathbb{L}}
\newcommand{\Up}{\Upsilon_{\mathbf{P}}}
\newcommand{\Ud}{\Upsilon_{\mathbf{D}}}
\newcommand{\Sd}{S_{\mathbf{D}}}
\newcommand{\Qp}{Q_{\mathbf{P}}}
\newcommand{\Cp}{C_{\mathbf{P}}}
\newcommand{\Rp}{R_{\mathbf{P}}}

\providecommand{\rset}[1]{\mathbb{R}^}

\DeclareMathOperator{\diag}{diag}
\DeclareMathOperator{\rank}{rank}

\DeclareMathOperator{\opd}{d\!}
\DeclareMathOperator{\re}{Re}

\usepackage[colorinlistoftodos,bordercolor=orange,backgroundcolor=orange!20,linecolor=orange,textsize=scriptsize]{todonotes}

\usepackage{array}
\newcolumntype{M}[1]{>{\centering\arraybackslash}m{#1}}

\newcommand{\rev}{\color{black}}

\begin{document}

\begin{frontmatter}

\vspace{-30pt}

\title{Model reduction with pole-zero placement\\ and {\rev high order moment matching}}

\thanks[footnoteinfo]{The research leading to these results has received funding from the NO Grants 2014 – 2021, under Project ELO-Hyp, contract no. 24/2020 and partially from the grant of the Romanian Ministry of Education and Research, CNCS - UEFISCDI,
project number PN-III-P1-1.1-TE-2019-1614, within PNCDI III. This paper has not been presented at any IFAC meeting. Corresponding author T.~C.~ Ionescu.}

\vspace*{-20pt}
\author[Ionescu,IonescuISMMA]{Tudor C. Ionescu}\ead{tudor.ionescu@upb.ro}, 
\author[Iftime]{Orest V. Iftime}\ead{o.v.iftime@rug.nl}, 
\author[Ionescu,IonescuISMMA]{Ion Necoara}\ead{i.necoara@upb.ro}

\address[Ionescu]{Department of Automatic Control and Systems Engineering, University Politehnica Bucharest, 060042 Bucharest, Romania.}
\address[IonescuISMMA]{Gheorghe Mihoc-Caius Iacob Institute of Mathematical Statistics and Applied Mathematics of the Romanian Academy, 050711 Bucharest, Romania.}
\address[Iftime]{Faculty of Economics and Business, University of Groningen, 9747AE Groningen, The Netherlands.}
\vspace{-10pt}
                                          
\begin{abstract}   
In this paper, we compute a low order approximation of a system of large order $n$ that matches $\nu$ moments of order $j_i$ of the transfer function, at $\nu$ interpolation points, has $\ell$ poles and $k$ zeros fixed and also matches $\nu-(\ell +k)$ {\rev moments of order $j_i+1$, where $j_i+1$ is the multiplicity of the $i$-th interpolation point}. We derive explicit linear systems in the free parameters to simultaneously achieve the required pole-zero placement and match the desired {\rev high order moments}.  We compute the closed form of the free parameters that meet the constraints, {\rev as the solution of a $\nu$ order linear system}. Furthermore, for data-driven model reduction, we generalize the {\rev construction of the Loewner matrices to include the data and the imposed pole and higher order moment} constraints. The resulting approximations achieve a trade-off between the good norm approximation and the preservation of the dynamics of the original system in a region of interest. 
\end{abstract}

\end{frontmatter}

\section{Introduction}
\label{intro}

Since the mathematical models of physical and industrial plants are highly dimensional {\rev linear, time-invariant (LTI) systems}, model reduction is  called for to find low-order approximations that meet desired constraints. Moment matching-based approximation techniques stand out as computationally efficient and easy to implement  \cite{antoulas-2005}. Based on a time-domain and Sylvester equation approach to moment matching taken in \cite{astolfi-TAC2010,i-astolfi-colaneri-SCL2014}, the  notion of moment is related to the unique solution of a Sylvester equation, see also \cite{gallivan-vandendorpe-vandooren-JCAM2004,gallivan-vandendorpe-vandooren-MTNS2006}.  Families of $\nu$ order models parametrized in $\nu$ free parameters, that match a set of $\nu$ moments of order $j_i$ of a given $n$-th order system, at a set of $\nu$ interpolation points of multiplicity $j_i+1$, are computed. 


\textbf{Motivation and contributions.} Fixing \emph{all} the $\nu$ degrees of freedom provides the unique $\nu$ order model that meets a \emph{single} required constraint. For instance, in \cite{astolfi-TAC2010} the $\nu$ free parameters are selected such that stability or the relative degree are preserved. 
{\rev In \cite{i-astolfi-colaneri-SCL2014} all the $\nu$ degrees of freedom are used to compute the \emph{unique} reduced model of \emph{minimal order} $\nu/2$ that matches $\nu$ moments of the given system. Furthermore, \cite{ionescu-iftime-ACC2012} addresses the computation of families of stable, LTI $\nu$ order models for infinite-dimensional systems, using state-feedback stabilization arguments without a closed form of the $\nu$ degrees of freedom.
}
In \cite{anic-beattie-gugercin-antoulas-AUT2013,gugercin-antoulas-beattie-SIAM2008} matching $\nu$ zero and first order moments of the system at the mirror images of the $\nu$ poles of the approximant yields the model with the lowest approximation error $H_2$-norm. In \cite{i-TAC2016}  the model that matches $2\nu$ moments as well as the model that matches $\nu$ moments of the given system and $\nu$ moments of its first order derivative are computed. Recently, in \cite{NecIon:18,NecIon:19}, using  optimization algorithms, the model achieving the minimal $H_2$-norm of the approximation error has been found. %
{\rev Furthermore, in \cite{ibrir-IEEE2017}, optimization-based methods are used for minimizing  a mixed $H_2/H_\infty$ ``small-gain'' criterion yielding a local minimizer.}
%
%
{\rev However, all the aforementioned techniques \emph{inherently place} the the poles and/or zeros of the reduced order models obtained at arbitrary locations in the complex plane, e.g., close to the imaginary axis, losing practical desired behaviours. The methods either focus on the placement of poles such that constraints of stability are met or such that the approximation error is minimized. We motivate the work of this paper with the trade-off between constraints of stability, practical significance and approximation accuracy. We place a number of poles at prescribed locations and match a number of high(er) order moments for decreasing the approximation error. We extend the arguments in \cite{datta-chakraborty-chaudhuriTAC2012}, where the state-feedback controller is designed to place some poles of the plant while the rest are constrained at the given locations.} However, in model reduction, an approximation satisfying \emph{multiple} properties, such as fixing stable poles and zeros and matching  moments is a yet unsolved problem.

In this paper, we seek a $\nu$ order approximation that simultaneously satisfies \emph{multiple} properties, i.e., matches $\nu$ moments of order $j_i$ at $\nu$ interpolation points of multiplicity $j_i+1$ and has $\ell$ poles, $k$ zeros and matches $\nu-(\ell+k)$ moments of order $j_i+1$, $i=1:l$, such that $\sum_i j_i=\nu$. 
We provide a linear system that yields the sufficient condition on the free parameters to place $\ell\leq\nu$ poles. For a particular canonical form of the interpolation points, we write the necessary and sufficient condition on the free parameters for the pole placement. 
%
%
We also derive the explicit linear system necessary and sufficient to place $k\leq\nu$ zeros. Moreover, we write the explicit linear system such that $\nu-(\ell+k)$ moments of order $j_i+1$ are matched. Then, in the framework of data-driven model reduction, we solve the problem of finding a reduced order model that matches the given data and satisfies the pole and moment constraints simultaneously.  Generalizing the Loewner matrices presented in \cite{gosea-zhang-antoulas2020,mayo-antoulas-LAA2007} for model reduction and in, e.g., \cite{Kergus-Formentin-PoussotVassal-Demourant-ECC2018} for control, we compute the closed form of the free parameters that meet {\emph all} the constraints. 
The resulting reduced order models achieve a trade-off between the good norm approximation and the preservation of the given dynamics in a region of interest. 


\textbf{Content.} In Section \ref{sect_prel} we recall the time-domain moment matching for linear systems. In Section \ref{sect_pzd_place}, we give and solve the sets of linear constraints to place poles, zeros and match further moments, respectively. In Section \ref{sect_Loewner_PZD}, we include all the constraints in the Loewner matrices to compute the explicit formula of the parameters satisfying the constraints. In Section \ref{sect_expl} we illustrate the theory on {\rev a cart controlled by a double pendulum and on} a CD player.


\textbf{Notation.} $\mathbb{R}$ is the set of real numbers and $\mathbb{C}$ is the set of complex numbers. $\mathbb{C}^{-}$ denotes the set of complex open left half plane. If $A$ is a matrix, then $A^{T}$ is the transpose. $\sigma(A)$ is the spectrum of  $A$. 


\section{Preliminaries}
\label{sect_prel}
\noindent Consider a single input-single output (SISO) linear time-invariant  (LTI) minimal  system
\begin{equation}\label{system}
\begin{split}
&\Sigma: \;\; \dot x = Ax+Bu, \quad y=Cx,
\end{split}
\end{equation}


with the state $x\in\mathbb{R}^n$, the input $u\in\mathbb{R}$ and the output $y\in\mathbb R$. The transfer function of \eqref{system} is
\begin{equation}\label{tf}
K(s)=C(sI-A)^{-1}B,\quad K:\mathbb{C} \to \mathbb {C}.
\end{equation}


Throughout the rest of the paper we assume that the system \eqref{system} is stable, that is $\sigma(A) \subset \mathbb{C}^-$, controllable and observable. 
{\rev Furthermore, $\sigma(A)$ is assumed a symmetric set of complex numbers, if for all $\lambda\in\sigma(A)$ then $\bar{\lambda}\in\sigma(A)$.}
\\
Then, the moments of (\ref{tf}) are defined as follows.
\begin{deff}\label{def_moment}\n\cite{astolfi-TAC2010}
{\rev The 0-moment of $K(s)$ at $s_1\in\mathbb C\setminus \sigma(A)$ is $\eta_0(s_1)=K(s_1)\in\mathbb C$.}
The $j_i$-moment
of $K$ from \eqref{tf} at $s_{i}\in\mathbb C\setminus\sigma(A)$ with multiplicity $j_i+1$, $i=1:l$, is defined by
$\eta_{j_i}(s_{i})={(-1)^{j_i}}/{(j_i!)}\left[{\opd^{j_i}K(s)}/{\opd s^{j_i}}\right]_{s=s_i}\in\mathbb C.$
\end{deff}


Let $s_i\in\mathbb C\setminus\sigma(A)$, $i=1:l$,  be a symmetric set of complex numbers. Take $j_i\geq 0$
such that
$\sum_{{\rev i=0}}^l j_i=\nu.
$
For each $i$, let $\eta_0(s_i),...,\eta_{j_i}(s_i)$ denote the $\nu$ moments of order $j_i+1$ of \eqref{system} at the given points $s_i$. 
Let $S\in\mathbb{C}^{\nu\times\nu}$, with the symmetric spectrum $\sigma(S)=\{s_i\ |\ i=1:l\}$, be such that $\sigma(S)\cap\sigma(A)=\emptyset$. Let $L\in\mathbb{C}^{1\times\nu}$ be such that the pair $(L,S)$ is observable. %
%
%
Let $\Pi\in\mathbb{R}^{n\times\nu}$ be the solution of the Sylvester equation


\begin{equation}\label{eq_Sylvester_Pi}
A\Pi+BL  =  \Pi S.
\end{equation}


Furthermore, since the system is minimal, assuming that $\sigma(A)\cap\sigma(S)=\emptyset$, then
$\Pi$ is the unique solution of the equation (\ref{eq_Sylvester_Pi})
and ${\rm rank}\ \Pi=\nu$, see e.g. \cite{desouza-bhattacharyya-LAA1981}.
Then, the moments of (\ref{system}) are characterized as follows.

\begin{prop}
\label{prop_mom_time}
\cite{astolfi-TAC2010}
\label{def_PI}
The $\nu$ moments $\eta_0(s_i),...,\eta_{j_i}(s_i), i=1:l$,  of system (\ref{system}) at $\sigma(S)$ are in one-to-one relation with the elements of the matrix $C\Pi$, i.e. the moments are uniquely determined by the elements of this matrix.
\end{prop}


Consider the LTI system
$
\dot{\xi}=F\xi+Gu,\ \psi=H\xi,
$
with $F\in\mathbb{C}^{\nu\times\nu},\ G\in\mathbb{C}^{\nu}$ and $ H\in\mathbb{C}^{p\times\nu}$, and the corresponding  transfer function
$
K_G(s)=H(sI-F)^{-1}G.
$
Let $\hat\eta_0(s_i),...,\hat\eta_{j_i}(s_i)$ denote the first $j_i+1$ moments of $K_G$ at $s_i$. Then, moment matching is defined as follows. 
\begin{deff}\label{deff_MM}\n\cite{iftime-i-ECC2013}
 $K_G$ matches $\nu$ moments of  $K$ at $\{s_1,...,s_l\}$, if
$
\eta_{\kappa}(s_i)=\hat\eta_{\kappa}(s_i),
$
for all $\kappa=0:j_i$, $i=1:l$.
\end{deff}
The next result gives the necessary and sufficient conditions
for a low-order system to achieve moment matching.

\begin{prop}
\cite{astolfi-TAC2010}
\label{prop_FGL}
Fix $S\in\mathbb{C}^{\nu\times\nu}$ and $L\in\mathbb{C}^{1\times\nu}$,
such that the pair $(L,S)$ is observable and $\sigma(S)\cap\sigma(A)=\emptyset$.
Furthermore, assume that $\sigma(F)\cap\sigma(S)=\emptyset$. Then, the reduced system $K_G$ matches the moments of (\ref{system}) at $\sigma(S)$ if and only if
$
HP=C\Pi,
$
where the matrix $P\in\mathbb{C}^{\nu\times\nu}$ is the
unique solution of the Sylvester equation $FP+GL=PS.$
\end{prop}

\noindent Note that, since the pair $(L,S)$ is observable, $P$ is invertible and $P^{-1}FP=S-{\rev P^{-1}}GL$. We are now ready to present the family of $\nu$ order models parameterized in $G$ that match $\nu$ moments of \eqref{system}. The system
\begin{equation}\label{redmod_CPi}
\Sigma_{G}: \,\dot{\xi}=(S-GL)\xi+Gu,\quad \psi=C\Pi\xi,
\end{equation}
with the transfer function
\begin{equation}\label{tf_redmod_CPi}
K_G(s)=C\Pi(sI-S+GL)^{-1}G,
\end{equation}
describes the family of $\nu$ order models that match $\nu$ moments at $\sigma(S)$ of \eqref{system}, in the sense of Definition \ref{deff_MM}, for all $G\in\mathbb R^{\nu}$ such that $\sigma(S-GL) \cap \sigma(S)=\emptyset$.

\begin{prob}\label{prob_all_constraints}\n
Consider \eqref{system} and the family of approximations $\Sigma_G$ as in \eqref{redmod_CPi}, matching $\nu$ moments of order $0:j_i$ of \eqref{system} at $s_i$, $i=1:l$, with multiplicity $j_i$. Find the parameter matrix $G$ such that
\begin{enumerate}[label=\roman*)]
\item $\Sigma_G$ has $\ell$ poles at $\lambda_1,\dots,\lambda_\ell$,
\item $\Sigma_G$ has $k$ zeros at $z_1,\dots,z_k$,
\item $\nu-(k+\ell)$ moments of order $1:j_i+1$ of $K_G$ and $K$ at $s_i$ match.
\end{enumerate}
\end{prob}


\section{Model reduction with pole-zero placement and  matching of {\rev high order moments}}\label{sect_pzd_place}

In this section we derive the linear constraints \eqref{eq_G_pole} and \eqref{eq_G_constraint_explicit}, \eqref{eq_G_zero_explicit} and \eqref{eq_Gd}, parametrized in $G\in\mathbb C^\nu$, resulting in the linear system \eqref{eq_constraints_all} to provide the solution to Problem \ref{prob_all_constraints}. 

\subsection{Pole placement as linear constraints}
\label{sect_p_place}

In this section, we place $\ell$ poles of the reduced order, for example in some of the poles of the original system, by properly selecting $G$.
Consider  \eqref{system} and the class of reduced $\nu$ order models $\Sigma_G$ from \eqref{redmod_CPi} that match $\nu$ moments of \eqref{system} at $\sigma(S)$. Let $\lambda_i\in\mathbb{C}$, $i=1:\ell$, $\ell\leq \nu$ be such that $\lambda_i\notin \sigma(S)$. Then $\lambda_i$ are poles of $\Sigma_G$ if $\det (\lambda_iI-S+GL)=0,\ i=1:\ell$ and such that $\{\lambda_1,\dots,\lambda_\ell\}$ is a symmetric set. 
To this end, let $\Qp\in\mathbb C^{\ell\times\ell}$ be a matrix such that $\sigma(\Qp)=\{\lambda_1,\dots,\lambda_\ell\}$. Furthermore, consider $\Cp \in\mathbb C^{1\times n}$ such that $\Cp \Pi=0$, where $\Pi$ solves \eqref{eq_Sylvester_Pi}, and let $\Up\in \mathbb R^{\ell\times n}$ be the unique solution of the Sylvester equation
\begin{equation}\label{eq_Sylvester_Y}
\Qp\Up=\Up A+\Rp \Cp,
\end{equation}
with $\Rp\in\mathbb R^\ell$ any matrix such that the pair $(\Qp,\Rp)$ is controllable%
%
%
Hence $\rank\Up=\ell$, see. e.g., \cite{desouza-bhattacharyya-LAA1981}. The next result imposes linear constraints on  $G$ such that  the reduced model $\Sigma_G$ has $\ell$ poles at $\{\lambda_1,\dots,\lambda_\ell\}$.

\begin{thm}\label{thm_G_pole}
Let $\Sigma_G$ as in \eqref{redmod_CPi} be a $\nu$ order model that matches the moments of \eqref{system} at $\sigma(S)$. Furthermore, let $\Up\in \mathbb R^{\ell\times n}$  be the unique solution of \eqref{eq_Sylvester_Y} and assume that $\rank (\Up\Pi)=\ell$. 
{\rev Consider $\Cp \in\mathbb C^{1\times n}$ such that $\Cp \Pi=0$ (i.e., $C_p^T\in\ker \Pi^T$).}
If $G$ is a solution of the equation
\begin{equation}\label{eq_G_pole}
\Up\Pi G=\Up B,
\end{equation}
then $\sigma(\Qp)=\{\lambda_1,\dots,\lambda_\ell\}\subseteq\sigma(S-GL)$. 

\end{thm}
\vspace*{-10pt}
\begin{proof}
Let $\lambda\in\sigma(\Qp)$. Then there exists the (left) eigenvector $v\in\mathbb R^\nu$, $v\ne 0$, such that $v^T(\lambda I-\Qp)=0$. Post multiplying with $\Up\Pi$ yields 
$v^T(\lambda\Up\Pi-\Qp\Up\Pi)=0.$
Hence, by \eqref{eq_Sylvester_Y}, we write
$
v^T(\lambda\Up\Pi-\Up A\Pi-\Rp\Cp\Pi)=0.
$
Since assuming $\Cp\Pi=0$ leads to
$v^T(\lambda\Up\Pi-\Up A\Pi)=0,$
using \eqref{eq_Sylvester_Pi} further yields
$
v^T(\lambda \Up\Pi - \Up\Pi S+\Up BL)=0.
$
Assuming \eqref{eq_G_pole} holds, we get 
$
v^T\Up\Pi(\lambda I-S+GL)=0.
$
Since we assume that $\rank(\Up\Pi)=\ell$, then $(\Up\Pi)^Tv=0$ if and only if $v=0$. Hence, $\lambda\in\sigma (S-GL)$ with the (left) eigenvector $(\Up\Pi)^T v$ and the claim follows.\fin
\end{proof}
\vspace*{-10pt}
\begin{obs}\n
Theorem \ref{thm_G_pole} yields the sufficient condition \eqref{eq_G_pole} on $G$ such that $\ell\leq\nu$ of the poles of $\Sigma_G$ are fixed, when $S,L$ and $\Qp$ are arbitrary matrices such that the pair $(L,S)$ is observable and the pair $(\Qp,\Rp)$ is controllable. Furthermore, if $\ell=\nu$ and $\Up\Pi$ is assumed invertible, then $\sigma(S-GL)=\sigma(\Qp)$, if and only if 
$
G=(\Up\Pi)^{-1}\Up B.
$
Note that $\Up$ and $\Pi$ can be easily computed explicitly using, e.g., Krylov projections and a coordinate transformation, to avoid solving any Sylvester equation. Furthermore, a sufficient condition to satisfy \eqref{eq_G_pole} is to select $G$ as a solution of the matrix equation $\Pi G=B$. Hence, post-multiplying equation \eqref{eq_Sylvester_Y} with $\Pi$ yields $\Qp\Up\Pi=\Up A\Pi$. Using equation \eqref{eq_Sylvester_Pi} one immediately gets $\Up A\Pi=\Up\Pi (S-GL)$. Moreover, if $\Up \Pi$ is assumed invertible, then the $\nu$ order model $\Sigma_G$ with $G$ such that $\Pi G=B$ is written equivalently as 
$(\Up\Pi)^{-1}\Up A\Pi = S-GL,\quad G=(\Up\Pi)^{-1}\Up B.$
\end{obs}
%


When $S$ is chosen diagonal   and the zero-order moments are considered,  \eqref{eq_G_pole} can be replaced by  a linear  system in the unknown $G\in\mathbb{R}^\nu$ without solving Sylvester equations. The next result gives a necessary and sufficient condition to place $\ell\leq\nu$ poles.
\begin{prop}
\label{prop_G_pole_explicit} Let $S=\diag\{s_1,\dots,s_\nu\}$ and $L=[1\ \dots\ 1]\in\mathbb R^{1\times\nu}$. Then  $\{\lambda_1,\dots,\lambda_\ell\}$ are poles of $K_G(s)$ as in \eqref{tf_redmod_CPi} if and only if $G\in\mathbb{C}^\nu$ is the solution of the linear system
\begin{equation}\label{eq_G_constraint_explicit}
1+LD_k^{-1}G=0,\quad \forall k=1:\ell,
\end{equation}
with $D_k=\diag(\theta_{k1},\dots,\theta_{k\nu})$, where $\theta_{ki}=\lambda_k-s_i,\ i=1:\nu$ and  $k=1:\ell.$
\end{prop}

\vspace{-10pt}

\begin{proof}
Note that $\lambda_i\in\mathbb C$ is a pole of $K_G(s)$ from \eqref{tf_redmod_CPi} if $\det (\lambda_iI-S+GL)=0$. Explicitly writing the determinant yields the equivalent equation
\vspace{-10pt}
\begin{equation*}\label{eq_sys_G_pole_canonic}
\begin{split}
& \left| \begin{matrix} \theta_{k1}+g_1 & g_1 & \dots & g_1 \\ g_2 & \theta_{k2}+g_2 & \dots & g_2 \\ \vdots & \vdots & \ddots & \vdots \\ g_\nu & g_\nu & \dots & \theta_{k\nu}+g_\nu\end{matrix}\right|=0,\\& \theta_{ki}=\lambda_k-s_i,\ i=1:\nu,\ k=1:\ell
\end{split}
\end{equation*}
and equivalently, in matrix form
$
\det (D_k+GL) = 0,
$
where $D_k=\diag(\theta_{k1},\dots,\theta_{k\nu})$, for each $k=1:\ell$. Using the well-known Sherman-Morrison-Woodbury formula \cite{horn-johnson-1985}, the claim follows immediately.\fin
\end{proof}

\vspace{-10pt}



\subsection{Zero placement as linear constraints}\label{sect_z_place}


Consider a system \eqref{system} and the family of $\nu$ order models $\Sigma_G$ that approximate \eqref{system} by matching $\nu$ moments, for all $G\in\mathbb R^\nu$. Let $z_1, ...,\ z_k\in\mathbb{R}$, $k\leq \nu$. By, e.g., \cite{astolfi-TAC2010,ionescu-iftime-ACC2012,iftime-i-ECC2013}, there exists a subfamily of models $\Sigma_G$, with the property that the set of zeros of each model contains $z_1, ...,\ z_k$. Equivalently, there exists $G$ such that
\begin{equation}\label{eq_assign_zeros}
\det\left[\begin{array}{cc} z_iI- S &  G \\ C\Pi & 0\end{array}\right]=0,\quad i=1:k.\end{equation}

Now, let $G=[g_1\ g_2\ \dots\ g_\nu]^T\in\mathbb R^\nu$. Then, it follows immediately that condition \eqref{eq_assign_zeros} is equivalent to a system of $k$ equations with $\nu$ unknowns $g_1,\dots, g_\nu$, given by
\begin{align*}\label{eq_G_zero}
(-1)^\nu\left[-g_1\zeta_1(z_1)+g_2\zeta_2(z_1)+\dots +(-1)^{\nu}g_\nu\zeta_\nu(z_1)\right] &=0, \nonumber \\
&  \vdots \\
(-1)^\nu\left[-g_1\zeta_1(z_k)+g_2\zeta_2(z_k)+\dots +(-1)^{\nu}g_\nu\zeta_\nu(z_k)\right] &=0, \nonumber
\end{align*}
with $\zeta_j(s)$ polynomials of degree $\nu-1$, $j=1:\nu$.
Note that when $S$ is diagonal and $L$  the polynomial equations can be replaced by a linear system in the unknown $G\in\mathbb{C}^\nu$.

\begin{prop} 
\label{prop:zeros}
Let $S=\diag\{s_1,\dots,s_\nu\}$, $L=[1\ \dots\ 1]$ and explicitly write the moments $C\Pi=[\eta_1\ \dots\ \eta_\nu]$.  Then $\Sigma_G$ as in \eqref{redmod_CPi} is a model which has $\{z_1,\dots,z_k\}$ among the zeros of the transfer function $K_G(s)$ given by \eqref{tf_redmod_CPi}, if and only if the elements of the matrix $G=[g_1\ \ g_2\ \dots\ g_\nu]^T$ satisfy the linear equations
\begin{equation}\label{eq_G_zero_explicit}
\sum_{i=1}^\nu \frac{\eta_i}{\gamma_{ji}}g_i=0,\quad j=1:k,
\end{equation}
where  $\gamma_{ji}=z_j-s_i,\ i=1:\nu,\ j=1:k.$
\end{prop} 


\begin{proof}
Note that  $\{z_1,\dots,z_k\}$ are zeros of $K_G(s)$ if and only if \eqref{eq_assign_zeros} is satisfied, i.e.,
\begin{equation*}\label{eq_sys_G_zero_canonic}
\begin{split}
&\left| \begin{matrix} \gamma_{j1} & 0 & 0 & \dots & 0 & g_1 \\ 0 & \gamma_{j2}& 0 & \dots & 0 & g_2 \\ \vdots & \vdots & \ddots & \vdots& \vdots & \vdots\\ 0 & 0 & 0 &\dots & \gamma_{j\nu}& g_\nu \\ \eta_1 & \eta_2 &\eta_3 & \dots & \eta_\nu & 0 \end{matrix}\right|=0,\\& \gamma_{ji}=z_j-s_i,\ i=1:\nu,\ j=1:k.
\end{split}
\end{equation*}
First, note that $\gamma_{ji}\ne 0$ for all $i,j$.
Successively decomposing the determinant by the last column and computing the resulting minors performing row decomposition yields
$
\sum_{i=1}^\nu \eta_i g_i \prod_{l=1:\nu,l\ne i}\gamma_{jl}=0,\quad j=1:k.
$ 
Dividing by $\prod_{l=1:\nu}\gamma_{jl}\ne 0, j=1:k$ leads to the claim.\fin
\end{proof}
%



\subsection{Matching high order moments as linear constraints}\label{sect_partial_deriv}


{\rev In this section we explicitly determine the matrix $G\in\mathbb R^\nu$ yielding the subfamily of models that match $\nu$ moments of order $0:j_i$ and $\mu\leq\nu$ moments of order $1:j_i+1$ of $K(s)$ at $\sigma(S)$. Without loss of generality, let $S=\diag(S_p,\Sd), \Sd\in\mathbb R^{\mu\times\mu}$.
We now use the state space representations of $K$ and $K'=\opd K/\opd s$, respectively. Let $L=[L_1^T\ L_2^T], L_2\in\mathbb{C}^{\mu}$. Let $\Pi=[\Pi_{1},...,\Pi_{\nu}]$ be the unique
solution of the Sylvester equation (\ref{eq_Sylvester_Pi}) and
$\Ud$ be the unique solution of the Sylvester equation
\begin{equation}\label{eq_Sylvester_Yd}
\Sd\Ud=\Ud A+RC,
\end{equation} 


with $R=L_2$. We assume that the pair $(\Sd,R)$ is controllable such that $\rank \Ud=\mu$. Then the moments of order $1:j_i+1$ of $K(s)$ at $\sigma(\Sd)$, are the moments of order $0:j_i$ of  $K'(s)$ at $\sigma(\Sd)$,  given by the elements of the matrix $\Ud\Pi\in\mathbb{C}^{\mu\times\nu}$, see \cite{i-TAC2016}. Then define 
$
\Sigma':\,\dot{x}=Ax+Bu,\ \dot{z}=Az+x,\ y=-Cz,
$
where $z\in\mathbb{R}^{n}$ and $y\in\mathbb{R}$, with the transfer function $K'$. Interconnecting $\Sigma'$ to the signal generator
\begin{equation}
\label{signal_gen-1}
\dot\omega=S\omega,\quad \theta=L\omega,\ \omega(0)\ne 0,\ \omega\in\mathbb R^\nu
\end{equation}
by $u=\theta$ and to the generalized signal generator
\begin{equation}
\label{gen_signal_gen}
\dot{\varpi}=\Sd\varpi+Rw, d=\varpi+\Ud z, \varpi(0)=0, \varpi\in\mathbb{R}^{\mu},
\end{equation}
by $w=y$, where $\Ud$ is the unique solution of \eqref{eq_Sylvester_Yd} and $R=L_2$, yields the
output signal $d(t)$. Then, by \cite[Theorem 2]{i-TAC2016}, the $0:j_i$ order moments of $K'$  at $\sigma(\Sd)$ are given by the steady-state behaviour of the signal $d$.}

We now impose  matching properties at the first order derivative of $K(s)$ in the sense of matching the relation defining signal $d(t)$. 
Consider a model $\Sigma_G$ as in \eqref{redmod_CPi} with the transfer function $K_G(s)$ given by \eqref{tf_redmod_CPi} and the state-space representation of $K'_G(s)$ as in \cite{i-TAC2016},
$
\Sigma'_G\!\!:\dot{\xi}=(S-GL)\xi+Gu,\, \dot{\chi}=(S-GL)\chi+\xi,\, \eta=-C\Pi\chi,
$
with $\chi(t)\in\mathbb{R}^{\nu}$. Considering the interconnection of ${\Sigma}'_G$ to the signal
generators $\dot\omega=S\omega,\quad \theta=L\omega,\ \omega(0)\ne 0,\ \omega\in\mathbb R^\nu$
by $u=\theta$ and
$
\dot{\varpi}=\Sd\varpi+Rw, d=\varpi+\Ud z, \varpi(0)=0, \varpi(t)\in\mathbb{R}^{\mu},$ by $v=\tilde{\eta}$, respectively, yields the output $\zeta(t)=\varpi(t)+{P}\chi(t).$ We say that the moments of order $j_i+1$ of $K_G$ match the $j_i+1$ moments of $K$ at $\sigma(\Sd)$ if the dynamics of $\zeta(t)$ are similar to the dynamics of $d(t)$ from \eqref{gen_signal_gen}, i.e., $\dot{\zeta}=\Sd\zeta+\Ud\Pi\xi,$ with $\Ud$ the solution of \eqref{eq_Sylvester_Yd} and $\Pi$ the solution of \eqref{eq_Sylvester_Pi}. The next result presents the closed form of $G\in\mathbb C^\nu$ such that $K_G$ matches $\nu$ moments of order $j_i$ of $K$ at $\sigma(S)$ and $K_G$ matches $\mu\leq\nu$ moments of order $j_i+1$ of $K$ at $\sigma(\Sd)$.
\begin{thm}\label{thm_match_deriv} Let $\Pi$ be the unique solution of
(\ref{eq_Sylvester_Pi}) and $\Ud$ be the unique solution
of (\ref{eq_Sylvester_Yd}). Consider a model $\Sigma_G$ as in \eqref{redmod_CPi}.  {\rev Then the $\mu$ moments of order $1:j_i+1$ of $K_G$ at $ \sigma(\Sd)$ match the $\mu$ moments of order $1:j_i+1$ of $K$ at $\sigma(\Sd)\subset\sigma(S),$ if and only if}
\begin{equation}\label{eq_Gd}
\Ud\Pi G= \Ud B.
\end{equation}
\end{thm}
%
%
\begin{proof} We first prove the necessity. Since $\zeta=\varpi+P\chi$, then $\dot\zeta=\dot\varpi+P\dot\chi.$ The moments of $\Sigma'_G$ match the moments of $\Sigma'$ at $\sigma(\Sd)$ if $\dot\zeta=\Sd\zeta+\Ud\Pi\xi$. Hence, since $\dot{\varpi}=\Sd\varpi+Rw$ and $w=\eta$ , where $\eta$ is the output of $\Sigma'_G$, we write $\Sd\varpi-RC\Pi\chi+P(S-GL)\chi+P\xi=\Sd\varpi+\Sd P\chi+\Ud\Pi\xi,$ for all $\xi$ and $\chi.$ Then, $P=\Ud\Pi$ and $PS-\Sd P=RC\Pi+PGL$. Equivalently, $\Sd\Ud\Pi-\Ud\Pi S=\Ud\Pi GL+RC\Pi$. Hence $\Ud\Pi GL=\Sd\Ud\Pi-\Ud\Pi S-RC\Pi$. By \eqref{eq_Sylvester_Yd}, $Q\Ud\Pi=(\Ud A+RC)\Pi$. Then, $\Ud\Pi GL=\Ud\Pi S-\Ud A\Pi.$ By \eqref{eq_Sylvester_Pi}, $\Ud A\Pi = \Ud (\Pi S-BL)$ and then the claim follows. Since the sufficiency uses similar arguments, the proof is omitted.\fin
\end{proof}
\vspace*{-10pt}
\begin{obs}\n
If $\mu=\nu$, the result in  \cite{i-TAC2016} is a particular case of \eqref{eq_Gd}. Hence, all the $\nu$ moments of order $1:j_i+1$ of $K_G(s)$ are matched at $\sigma(\Sd)=\sigma(S)$, by selecting
$
G=(\Ud\Pi)^{-1}\Ud B,
$
with $\Ud\Pi\in\mathbb{C}^{\nu\times\nu}$ assumed invertible.
\end{obs}

\subsection{The linear algebraic system of all constrains}


Let $\Sigma_G$ define a family of $\nu$ order models that match $\nu$ zero order moments  of \eqref{system} at $\{s_1,\dots,s_{\nu}\},$ parametrized in $G=[g_1\ \dots\ \ g_\nu]^T\in\mathbb{C}^\nu$, where $S$ is in diagonal form. Let $\{\lambda_1,\dots,\lambda_\ell\}\subset\mathbb{C}\setminus\sigma(S)$ and $\{z_1,\dots,z_k\}\subset\mathbb{C}$, $\ell+k\leq\nu$. Collecting the  linear constraints \eqref{eq_G_constraint_explicit}, \eqref{eq_G_zero_explicit} and \eqref{eq_Gd} yield the following system of linear equations in $G$:
%
%
\begin{equation}\label{eq_constraints_all}
\begin{cases}
1+LD_k^{-1}G=0,\quad k=1:\ell, \\
\sum_{i=1}^\nu \frac{\eta_i}{\gamma_{ji}}g_i=0,\quad j=1:k, \\
\Ud\Pi G= \Ud B,
\end{cases}
\end{equation}
%
%
with $D_k=\diag(\theta_{k1},\dots,\theta_{k\nu}), \theta_{ki}=\lambda_k-s_i,\ i=1:\nu,\ k=1:\ell,$ $\gamma_{ji}=z_j-s_i,\ i=1:\nu,\ j=1:k$, $\Ud$ is the solution of \eqref{eq_Sylvester_Yd} and $\Pi$ is the solution of \eqref{eq_Sylvester_Pi}. Then, $G$ satisfying \eqref{eq_constraints_all} yields the $\nu$ order model $\Sigma_G$ with the properties


\begin{itemize}
\item $\Sigma_G$ has $\ell$ poles at $\lambda_1,\dots,\lambda_\ell$,
\item $\Sigma_G$ has $k$ zeros at $z_1,\dots,z_k$,
\item $\nu\!-k\!-\ell$ first order moments  of $K_G$ and $K$ match.
\end{itemize}


However, in practice, the models \eqref{system} are not known, motivating the extension of the results to the data-driven model order reduction using the Loewner matrices \cite{mayo-antoulas-LAA2007}. 



{\section{\rev Loewner matrices-based modeling with pole placement and {high order moment matching}}
\label{sect_Loewner_PZD}

In this section we construct a general version of the Loewner matrices given in \cite{mayo-antoulas-LAA2007}, {\rev from the data to match $\nu$ {\em zero} order moments of $K$ from \eqref{tf} at $\{s_1,s_2,\dots, s_\ell,s_{\ell+1},\dots, s_\nu\}$, place $\ell$ poles at $\{\lambda_1,\dots,\lambda_\ell\}$, with $\ell\leq\nu$ and $\lambda_i\ne s_j, i,j=1:\ell$} and match $\nu-\ell$ {\em first} order moments of $K$ at $\{s_{\ell+1},\dots, s_\nu\}$, i.e.,
\begin{subequations}\label{eq_Loewner_general}
\begin{align}
\Lo_{ij} & =\begin{cases} 
\cfrac{\G(\lambda_i)-\G(s_j)}{\lambda_i-s_j}, & i,j=1:\ell, \\ 
-\cfrac{\G(s_i)-\G(s_j)}{s_i-s_j}, & i\ne j=\ell+1:\nu,\\ 
-\G'(s_i), & i=j=\ell+1:\nu,
\end{cases} \label{eq_Loewner_gen}\\
\sLo_{ij} & = \begin{cases} 
-\cfrac{\lambda_i\G(\lambda_i)-s_j\G(s_j)}{\lambda_i-s_j}, &\!\!\!\!\! i,j=1:\ell,\\
-\cfrac{s_i\G(s_i)-s_j\G(s_j)}{s_i-s_j}, &\!\!\!\!\! i\ne j=\ell+1:\nu \\ 
-s_i\G'(s_i), &\!\!\!\!\! i=j=\ell+1:\nu. 
\end{cases} \label{eq_shifted_Loewner_gen}
\end{align}
\end{subequations}
Let 
\begin{equation}\label{eq_S_Lo}
S=\diag(s_1,\dots, s_\ell,s_{\ell+1},\dots, s_\nu)=\diag(S_p,\Sd), 
\end{equation} 
with $S_p=\diag(s_1,\dots,s_\ell), \Sd=\diag (s_{\ell+1},\dots, s_\nu)$ and let $L=[1\ 1\ \dots\ \ 1]=[L_1\ L_2]\in\mathbb R^{1\times\nu}, L_2\in\mathbb R^{1\times(\nu-\ell)}$. Furthermore, let 
\begin{equation}\label{eq_Q_Lo}
Q\!=\!\diag(\lambda_1,\dots,\lambda_\ell,s_{\ell+1},\dots, s_\nu)\!=\!\diag(\Qp,\Sd),
\end{equation} 
with $\Qp=\diag(\lambda_1,\dots,\lambda_\ell)$. Let $\Pi$ be the solution of the Sylvester equation \eqref{eq_Sylvester_Pi}, $A\Pi+BL=\Pi S$. Furthermore construct $\Upsilon=[\Up^T\ \Ud^T]^T\in \mathbb R^{\nu\times n}$, where $\Up$ is the unique solution of the Sylvester equation \eqref{eq_Sylvester_Y}, $\Qp\Up=\Up A+\Rp\Cp,$ where $\Cp\in\mathbb R^{1\times n}$ such that $\Cp\Pi=0$ and $\Ud$ is the unique solution of the Sylvester equation \eqref{eq_Sylvester_Yd}, $\Sd\Ud=\Ud A+RC$, where $R=L_2^T$. Note that, in matrix form, $\Upsilon$ is the unique solution of the Sylvester equation
\begin{equation}\label{eq_Sylvester_Yii}
Q\Upsilon=\Upsilon A+ \mathbf R(\Cp, C),
\end{equation}
where $\mathbf R(\Cp, C)=[(\Rp\Cp)^T\ (RC)^T]^T.$  The next result result states that the Loewner matrices given by equations \eqref{eq_Loewner_general} can be written directly in terms of $\Upsilon$ and $\Pi$ and that they are the solutions of two Sylvester equations. 
\begin{thm}\label{thm_Lo_YPi_gen}
Consider the Loewner matrices from \eqref{eq_Loewner_general} and the matrices $S$ and $Q$ defined by \eqref{eq_S_Lo} and \eqref{eq_Q_Lo}. Let $\Pi$ be the unique solution of \eqref{eq_Sylvester_Pi} and $\Upsilon$ be the unique solution of \eqref{eq_Sylvester_Yii} and assume that the matrix $\Upsilon\Pi$ is invertible. Consider the statements
\begin{enumerate}[label=(\arabic*)]
\item $\Lo$ is defined by equation \eqref{eq_Loewner_gen},
\item $\Lo=-\Upsilon\Pi$ and satisfies the Sylvester equation $\Lo S-Q\Lo=R \Cp \Pi-\Upsilon BL$,
\item $\sLo$ is defined by equation \eqref{eq_shifted_Loewner_gen},
\item $-\sigma\mathbb L=S-(\Upsilon\Pi)^{-1}\Upsilon BL$ and satisfies the Sylvester equation $\sLo S-Q\sLo=RC\Pi S-Q\Upsilon BL$.
\end{enumerate}
Then (1) $\Leftrightarrow$ (2) and (3) $\Leftrightarrow$ (4).
\end{thm}

\vspace{-10pt}

\begin{proof}
We first prove statement $(1)\Leftrightarrow (2)$. Note that \eqref{eq_Loewner_gen} can be equivalently written as 
\vspace*{-5pt}
\begin{align}\label{eq_Loewner_explicit}
\lambda_i \Lo-\Lo s_j &= C(\lambda_iI-A)^{-1}B-C(s_jI-A)^{-1}B, \nonumber\\
s_i \Lo-\Lo s_j &= C(s_iI-A)^{-1}B-C(s_jI-A)^{-1}B,
\end{align}
for all $i,j=1:\ell$ and for all $i\ne j=\ell+1:\nu$, respectively.
Note that, for any $\alpha\ne\beta\in\mathbb C$,
\vspace*{-5pt}
{\rev \begin{align*}
&\frac{C(\alpha I-A)^{-1}B-C(\beta I-A)^{-1}B}{\alpha-\beta} \\ & =\frac{C[(\alpha I-A)^{-1}-(\beta I-A)^{-1}]B}{\alpha-\beta} \nonumber\\
            & =\frac{C(\alpha I-A)^{-1}[\beta I-A-\alpha I+A](\beta I-A)^{-1}B}{\alpha-\beta} \nonumber \\
            &= -C(\alpha I-A)^{-1}(\beta I-A)^{-1}B.
\end{align*}
Hence, substituting $\alpha=\lambda_i, \beta=s_j$, for all $i,j=1:\ell$  and substituting $\alpha=s_i, \beta=s_j$, for all $i\ne j=\ell+1:\nu$,  yields
$
\Lo_{ij} = -C(s_iI-A)^{-1}(\lambda_jI-A)^{-1}B,\quad \forall i,j=1:\ell,$ and
$\Lo_{ij} = -C(s_iI-A)^{-1}(s_jI-A)^{-1}B,\quad \forall i\ne j=\ell+1:\nu.
$}
Moreover, by \eqref{eq_Loewner_gen},
$
\Lo_{ii} =-\G'(s_i)=C(s_iI-A)^{-2}B
            =C(s_iI-A)^{-1}(s_iI-A)^{-1}B, i=\ell+1:\nu.
$
Let $\Upsilon_i=C(s_iI-A)^{-1}$ and $\Pi_j=(\lambda_jI-A)^{-1}B$, $i,j=1:\nu$. It is straight forward that $\Upsilon=[\Upsilon_1^T\ \Upsilon_2^T\ \dots\ \Upsilon_\nu^T]^T\in\mathbb C^{\nu\times n}$ and $\Pi=[\Pi_1\ \Pi_2\ \dots\ \Pi_\nu]\in\mathbb C^{n\times\nu}$ are the (unique) solutions of the Sylvester equations \eqref{eq_Sylvester_Yii} and \eqref{eq_Sylvester_Pi}, respectively. Hence, $\Lo_{ij}$ as in \eqref{eq_Loewner_gen} can be written equivalently as
$
\Lo_{ij} =-\Upsilon_i\Pi_j,\quad\forall i,j=1:\ell\ \text{and}\ \forall i,j=\ell+1:\nu, i\ne j, $ and 
$\Lo_{ii}= -\Upsilon_i\Pi_i,\quad\forall i=\ell+1:\nu.
$
Furthermore, writing \eqref{eq_Loewner_explicit} for each $i,j$ yields the claim. The arguments for $\sLo$ are similar, hence omitted.\fin
\end{proof}
\vspace*{-10pt}
We now write the approximation $\Sigma_G$ matching $\nu$ zero order moments of \eqref{system}, $\ell$ pole constraints and $\nu-\ell$ first order moments of \eqref{system}, simultaneously.

\begin{thm}\label{thm_LisG} Let $\Sigma_G$ be a model described by the equations \eqref{redmod_CPi} with the transfer function \eqref{tf_redmod_CPi}. Then, for
\begin{equation}\label{eq_Gd_Loewner_all}
G=-\Lo^{-1}\Upsilon B,
\end{equation}
with $\Lo$ given by \eqref{eq_Loewner_gen} assumed invertible and $\Upsilon$ the solution of \eqref{eq_Sylvester_Yii}, the model $\Sigma_{-\Lo^{-1}\Upsilon B}$ matches $\nu$ zero order moments of \eqref{system} at $\sigma(S)=\{s_1,\dots, s_\nu\}$, has $\ell$ poles at $\{\lambda_1,\dots,\lambda_\ell\}\subset\sigma(Q)$ and matches $\nu-\ell$ first order moments of \eqref{system} at $\{s_{\ell+1},\dots,s_\nu\}\subset\sigma(S)$. Furthermore,
\begin{align}\label{eq_Loewner_approximant}
\G_{-\Lo^{-1}\Upsilon B}(s)=C\Pi(\sLo-s\Lo)^{-1}\Upsilon B.
\end{align}
\end{thm}

\vspace*{-15pt}

\begin{proof}
Employing Theorem \ref{thm_Lo_YPi_gen} and using \eqref{eq_G_pole}, \eqref{eq_Gd} and \eqref{eq_Loewner_general}, the claim follows immediately.
\end{proof}
\vspace*{-10pt}
{\rev \begin{cor}\label{rem_LisG}
For a $\Sigma_G$ as in \eqref{redmod_CPi}, the solution $G$ of the system \eqref{eq_constraints_all}, with $k=0$, is identical with $G$ from \eqref{eq_Gd_Loewner_all}.
\end{cor}
\vspace*{-10pt}
\begin{proof}
Using Theorems \ref{thm_G_pole}--\ref{thm_LisG} successively, one immediately recovers \eqref{eq_Gd_Loewner_all} and the claim follows.
\end{proof}}
\vspace*{-10pt}
In general, for any non-derogatory matrices $Q$ and $S$, with $R$ and $L$ such that the pair $(L,S)$ is observable and the pair $(Q,R)$ is controllable, the matrix $\widehat\Lo=-\Upsilon\Pi$, with $\Upsilon$ and $\Pi$ the unique solutions of \eqref{eq_Sylvester_Yii} and \eqref{eq_Sylvester_Pi}, respectively, satisfies the properties of a Loewner matrix.
\begin{thm}\label{cor_Loewner_any}
Consider  system \eqref{system}. Let $S\in\mathbb{R}^{\nu\times \nu}$ be any matrix with $\sigma(S)=\{s_1,s_2,\dots, s_\ell,s_{\ell+1},\dots, s_\nu\}$  not poles of \eqref{tf} and $L\in\mathbb{R}^{1\times\nu}$ such that the pair $(L,S)$ is observable. Also let $Q\in\mathbb R^{\nu\times \nu}$ be any matrix with $\sigma(Q)=\{\lambda_1,\dots,\lambda_\ell,s_{\ell+1},\dots, s_\nu\}$, not poles of \eqref{tf}. Furthermore, let $\Pi$ be the unique solution of the Sylvester equation \eqref{eq_Sylvester_Pi}, and $\Upsilon$ be the unique solution of \eqref{eq_Sylvester_Yii}. Then, the matrices
\begin{subequations}\label{eq_Loewner_any_all}
\begin{align}
\widehat\Lo & = -\Upsilon\Pi, \label{eq_Loewner_any}\\
\widehat\sLo & = S-(\Upsilon\Pi)^{-1}\Upsilon BL \label{eq_sLoewner_any}
\end{align}
\end{subequations}
satisfy the equations 
\begin{subequations}\label{eq_LoewnerSylv_any_all}
\begin{align}
\widehat\Lo S-Q\widehat\Lo    & = L^T\widetilde C\Pi-\Upsilon BL, \label{eq_LoewnerSylv_any}\\
\widehat\sLo S-Q\widehat\sLo & = L^TC\Pi S-Q\Upsilon BL.\label{eq_sLoewnerSylv_any}
\end{align}
\end{subequations}
Furthermore, $\widehat\Lo=T_Q^{-1}\Lo T_S,$ where $T_Q\in\mathbb R^{\nu\times\nu}$ is such that  $T_QQT_Q^{-1}=\diag(\lambda_1,\dots,\lambda_\ell,s_{\ell+1}, \dots, s_\nu)$ and $T_S\in\mathbb R^{\nu\times\nu}$ is such that $T_SST_S^{-1}=\diag(s_1,\dots, s_\ell,s_{\ell+1},\dots, s_\nu).$
\end{thm}


\begin{proof}
Pre-multiplying \eqref{eq_Sylvester_Pi} with $\Upsilon$ yields $\Upsilon A\Pi+\Upsilon BL=\Upsilon\Pi S$. By \eqref{eq_Sylvester_Yii}, $\Upsilon A=Q\Upsilon-\mathbf R(\widetilde C, C)$. Hence 
$
(Q\Upsilon-\mathbf R(\widetilde C, C))\Pi+\Upsilon BL=\Upsilon\Pi S
\Leftrightarrow Q\Upsilon\Pi-\Upsilon\Pi S=\mathbf R(\widetilde C, C)\Pi-\Upsilon BL,
$
which is equivalent to the Sylvester equation satisfied by $\mathbb L=-\Upsilon\Pi$ in \cite[equation (12)]{mayo-antoulas-LAA2007}. It follows that $\Upsilon\Pi(S-(\Upsilon\Pi)^{-1}\Upsilon BL)=(Q-\mathbf R(\widetilde C, C)\Pi(\Upsilon\Pi)^{-1})(\Upsilon\Pi)$. Hence, one can write $Q\widehat\Lo+\mathbf R(\widetilde C, C)\Pi=-S+(\Upsilon\Pi)^{-1}\Upsilon BL$. By \cite[Proposition 3.1]{mayo-antoulas-LAA2007}, the claim follows immediately. The second claim follows straightforwardly when applying the coordinate transformations $T_Q$ and $T_S$.\fin
\end{proof}}
%
\begin{table*}[!t]
\setlength{\tabcolsep}{3pt}
\centering
\begin{tabular}{|c|c|c|M{1.65cm}|c|c|M{1.65cm}|c|c|M{1.95cm}|c|c|}
\hline
$\nu$ &  $\ell$ & $k$ &Max $\re(p(K_G))$ & $\|K\!\!-\!\!K_G\|_{2}$ &  $K_G(0)$ & Max $\re(p(K_{\text{BT}}))$ &  $\|K\!\!-\!\!K_{\text{BT}}\|_{\infty}$& $K_{\text{BT}}(0)$ & Max $\re(p(K_{\text{IRKA}}))$ & $\|K\!\!-\!\!K_{\text{IRKA}}\|_{2}$& $K_{\text{IRKA}}(0)$   \\
\hline
\hline
\multirow{3}{*}{3} & 3 & 0 & $-7.4\cdot 10^{-1}$ & 1.523 &\multirow{10}{*}{4.5661}& \multirow{3}{*}{$-2.26\cdot 10^{-5}$} & \multirow{3}{*}{$6.6\cdot 10^{-2}$} & \multirow{3}{*}{$4.7206$} &\multirow{3}{*}{$-2.26\cdot 10^{-5}$} & \multirow{3}{*}{$2.09\cdot 10^{-3}$}& \multirow{3}{*}{$4.6395$} \\
\cline{2-5} & 2 & 1 & $-2.91\cdot 10^{-4}$ & 1.10 & & & & & & &\\
\cline{2-5} & 0 & 0 & $2\cdot 10^{-1}$     & $\infty$ & & & & & & &\\
\cline{1-5} \cline{7-12}
\multirow{3}{*}{6} & 6 & 0 & $-7.05\cdot 10^{-1}$ & 1.22 & &\multirow{3}{*}{$-2.26\cdot 10^{-5}$} & \multirow{3}{*}{$8.83\cdot 10^{-4}$}&\multirow{3}{*}{$4.6554$} & \multirow{3}{*}{$-2.26\cdot 10^{-5}$} &\multirow{3}{*}{$5.28\cdot 10^{-5}$}&\multirow{3}{*}{$4.657$}\\
\cline{2-5} & 4 & 2 & $-7.9\cdot 10^{-4}$ & 1.09 & & & & & & &\\
\cline{2-5} & 0 & 0 & $2.12\cdot 10^{-2}$ & $\infty$ & & & & & & &\\
\cline{1-5} \cline{7-12}
\multirow{3}{*}{12} & 12 & 0 & $-5.4\cdot 10^{-3}$ & 8.07 & &\multirow{3}{*}{$-2.26\cdot 10^{-5}$} & \multirow{3}{*}{$1.41\cdot 10^{-4}$}&\multirow{3}{*}{$4.653$} & \multirow{3}{*}{$-2.26\cdot 10^{-5}$} & \multirow{3}{*}{$1.37\cdot 10^{-5}$} &\multirow{3}{*}{$4.655$}\\
\cline{2-5} & 5  & 3 & $-1.86\cdot 10^{-6}$ & $4.16\cdot 10^{-3}$ & & & & & & &\\
\cline{2-5} & 0  & 0 & $8.64\cdot 10^{-1}$ & $\infty$ & & & & & & &\\
\hline
\end{tabular}
\caption{Simulation results for $\Sigma_G$ of order $\nu$ with $\ell$ poles and $k$ zeros and $\nu-(\ell+k)$ derivatives matched versus the BT and the IRKA.}
\label{tab_comp}
\vspace*{-15pt}
\begin{eqnarray*}
 \\
\hline
\end{eqnarray*}
\vspace{-20pt}
\end{table*}

\vspace{-10pt}

\section{Illustrative examples}\label{sect_expl}

{\rev
\subsection{Computational complexity}
The approximations that match $\nu$ zero order moments  at $s_i$, $i=1:\ell$,  place $\ell$ prescribed poles, $k$ prescribed zeros and match $\nu-(\ell+k)$ first order moments are computed employing \eqref{redmod_CPi} and then solving the linear system \eqref{eq_constraints_all} with complexity $\mathcal{O}(\nu^3)$. 
Using Theorem \ref{thm_LisG}, the storage of the $\nu\times\nu$ Loewner matrices and the inversion of $\Lo$ with complexity $\mathcal{O}(\nu^3)$ are required. The Sylvester equations involved can be solved efficiently using Krylov techniques, see, e.g., \cite{antoulas-2005} and the references therein. The simulations have been performed under Maple 2018 and Matlab R2015b, on a desktop equipped with 4GB RAM, 2.2MHz CPU and Windows 10.

\subsection{Cart controlled by a double pendulum}
\label{sect_expl_cart}

\noindent Consider the  cart system controlled by a double-pendulum controller, with $6$ states, with the matrices $A \in \mathbb{R}^{6 \times 6}, B \in \mathbb{R}^{6 \times 1}$ and $C \in \mathbb{R}^{1 \times 6}$, see \cite{i-TAC2016} for the explicit matrices.
The poles of the system are $\{-1.6 + 6.63j, -1.6 - 6.63j, -0.74 + 3.48j,-0.74 - 3.48j,-0.16 + 0.55j, -0.16 - 0.55j\}$, i.e., the system is stable.
Take the interpolation points 
$\{0, 1/4, 1/2\}
$.
\begin{figure}[!b]
\centering
\begin{subfigure}[b]{0.17\textwidth}
\includegraphics[width=\textwidth]{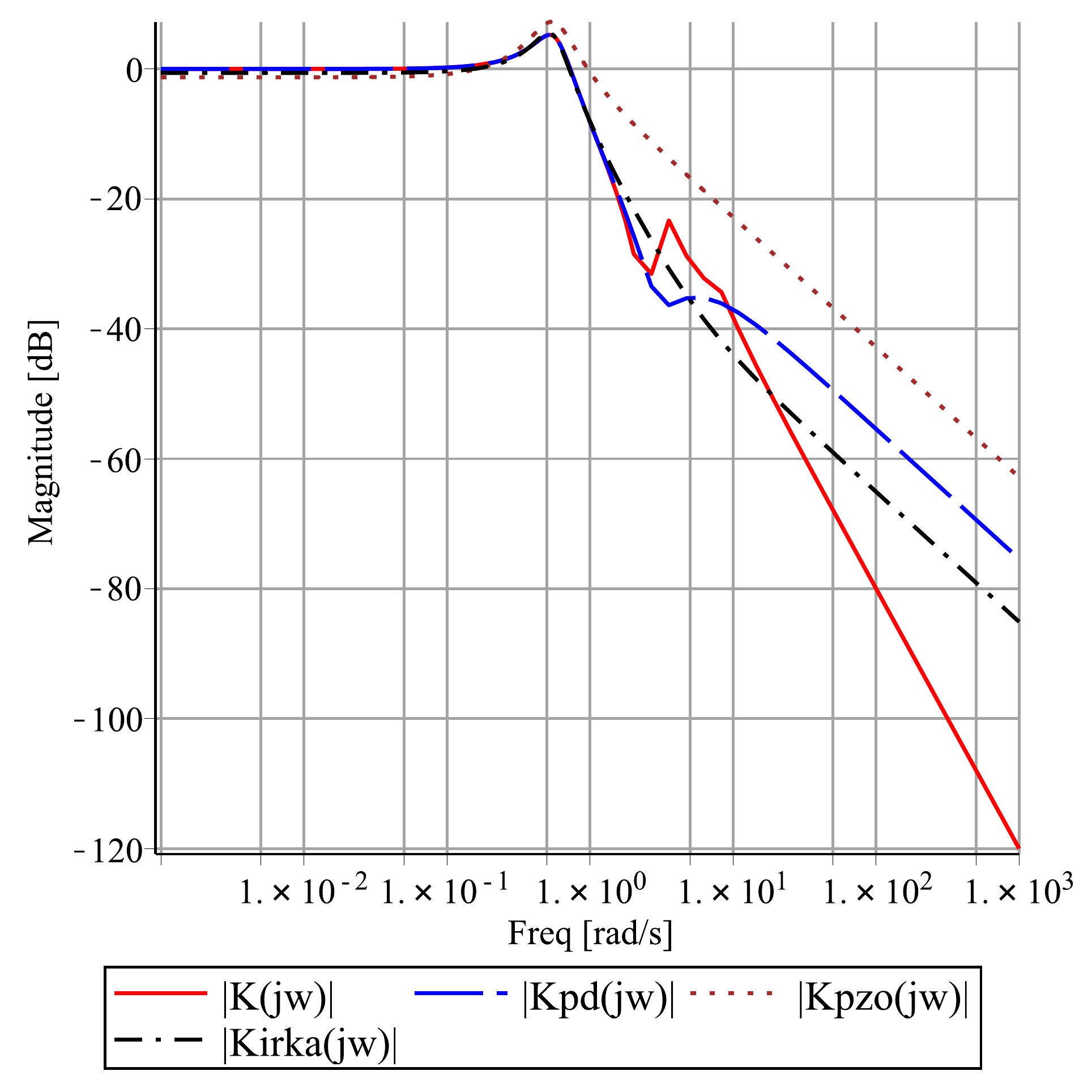}
\caption{The models}\label{fig_c2p_approx_Loewner}
\end{subfigure}
\hspace{0cm}
\begin{subfigure}[b]{0.17\textwidth}
\includegraphics[width=\textwidth]{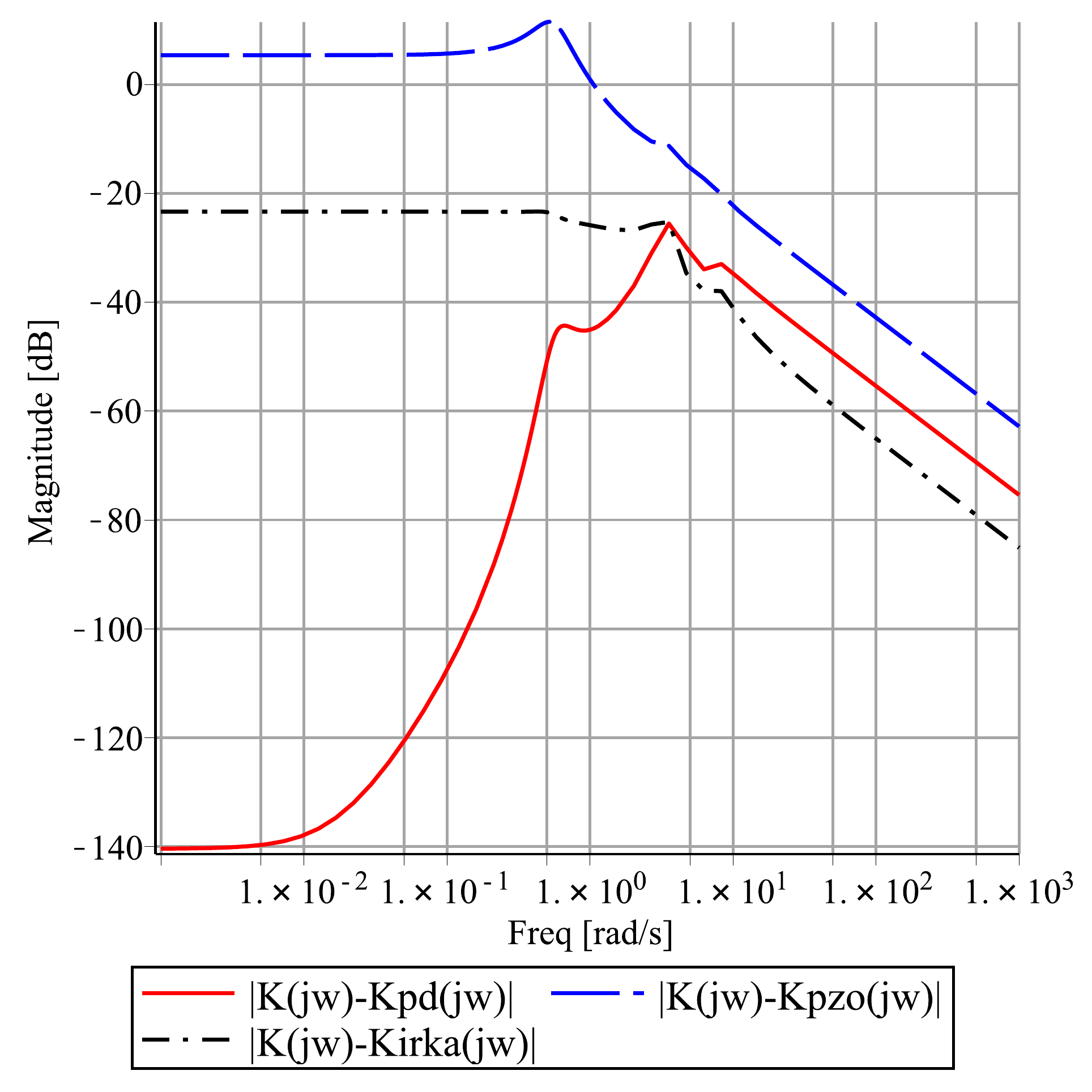}
\caption{Error systems}\label{fig_c2p_ERR_Loewner}
\end{subfigure}
\hspace{0cm}
\begin{subfigure}[c]{0.12\textwidth}
\vspace*{-105pt}
\includegraphics[width=\textwidth]{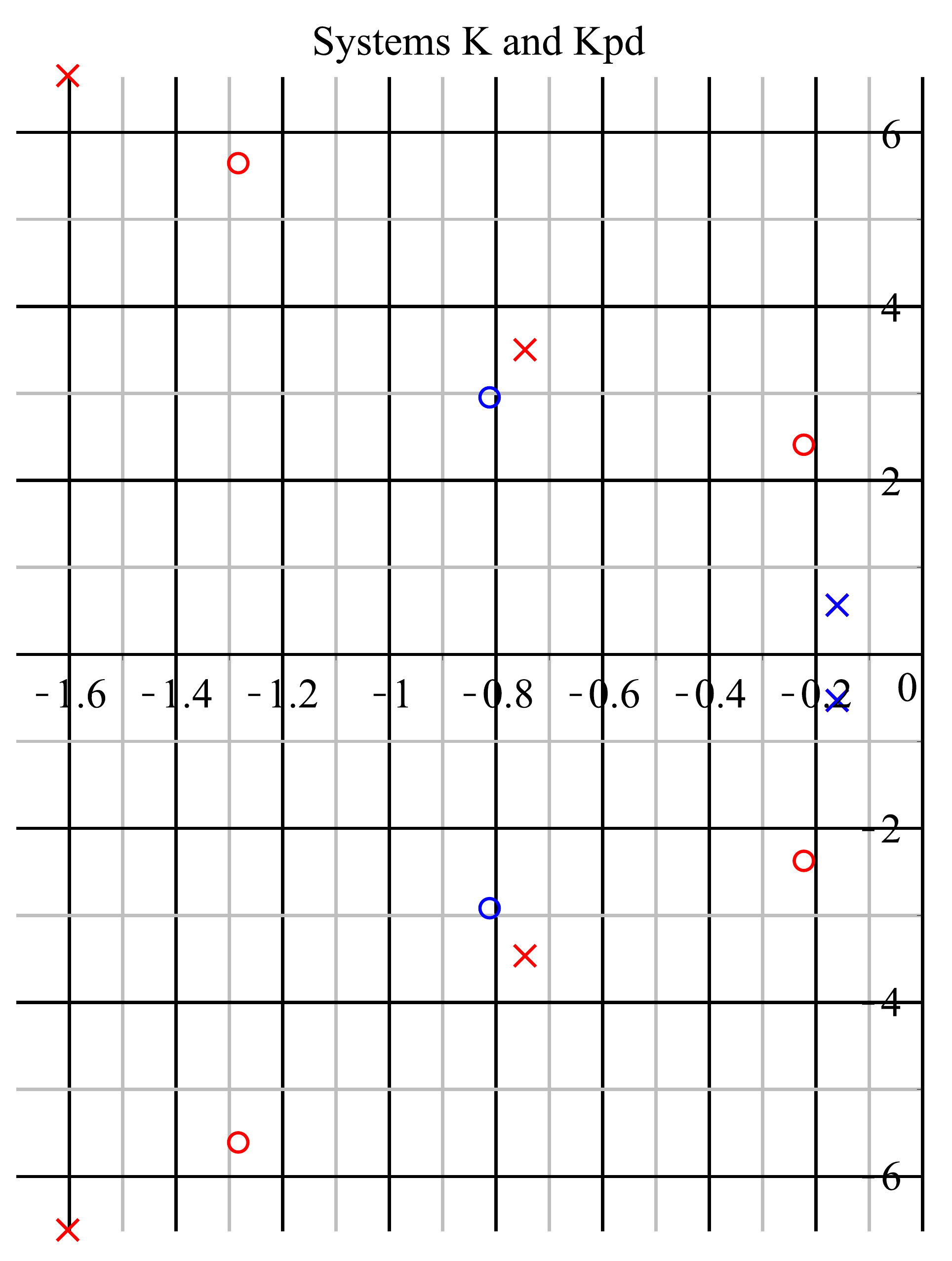}
\vspace{-2pt}
\caption{PZ map}\label{fig_c2p_zpk_L}
\end{subfigure}
\caption{Magnitude plots of the models \eqref{fig_c2p_approx_Loewner} and of the approximation errors \eqref{fig_c2p_ERR_Loewner} and the pole-zero map \eqref{fig_c2p_zpk_L} of the cart controlled by a double pendulum $K$ and the third order approximations $K_{\rm pd}, K_{\rm pzo}$ and $K_{\rm IRKA}$}\label{fig_c2p_L}
\end{figure}

Pick $L$ such that the pair $(L,S)$ is observable. Furthermore, using the solution $\Pi$ of \eqref{eq_Sylvester_Pi} yields $C\Pi=[1\   0.69\    0.45]$.  We build the third order model $K_{\rm pd}$ with the properties that $K_{\rm pd}$ matches the moments at $\sigma(S)=\{0,1/4,1/2\}$, $K'_{\rm pd}(0)=K'(0)$ and $K_{\rm pd}$ has poles at $-0.16\pm 0.55j$. Compute and store the Loewner matrices from \eqref{eq_Loewner_general}. Moreover, using $\Upsilon$, the solution of \eqref{eq_Sylvester_Y} one gets $\Upsilon B=[1\ 0.7 \ 0.5]^T$. Using \eqref{eq_Loewner_approximant} yields $K_{\rm pd}(s)=(0.3638s^2+0.58863s+3.37312)/(8.573s^3+44.283s^2+15.847s+13.493).$ The resulting $H_2$-error norm of the approximation is $1.41\cdot 10^{-1}$ and the approximant has two poles at $-0.16\pm 0.55j$. Using a family of third order models is described by ${\Sigma}_{G}$ as in (\ref{redmod_CPi}) and selecting
$G=-\Lo^{-1}\Upsilon B=[-12.591\ 0.992\ -31.3163]^T$ also leads to $K_{\rm pd}$.
We now compare with the third order model $K_{\rm pzo},$ yielded by the gradient method proposed in \cite{i-iftime-necoara-ECC2020}, with constraints of fixing a pole at $-1.6$ and a zero at $-1.28$, with the state-space realization given by \cite[eq. (35)]{i-iftime-necoara-ECC2020}, denoted by $K_{\rm pzo}$.
%
%
The $H_2$ norm of the approximation error achieved by $K_{\rm pzo}$ is $8.7\cdot 10^{-3}$. The poles of  $K_{\rm pzo}$ are $\{-1.6, -0.18 + 0.53j, -0.18 - 0.53j\}$ and the zeros are $\{-1.28, -45.92\}$, i.e., a stable and minimum phase third order approximation. Furthermore, we compare with the third order IRKA model, see, e.g., \cite{gugercin-antoulas-beattie-SIAM2008}. The IRKA is initialized in $\sigma(S)$. The resulting approximation is given by $K_{\rm IRKA}(s)=(0.05563s^2+0.2464s-0.3018)/(s^3-0.7795s^2-0.00163s-0.3238),$ with poles at $1.065983722, -0.1432418608\pm 0.5322017973j$ and the  $H_2$-error norm of order $7\cdot 10^-3$. 
Figure \ref{fig_c2p_approx_Loewner} shows that since moment matching at zero is imposed, all the approximations behave well at low frequency. The model $K_{\rm pd}$ exhibits almost identical responses to the harmonic inputs of frequencies up to approximately 6 rad/sec., whereas the rest preserve similar behaviours on smaller frequency sets, even if they appear more accurate. Figure \ref{fig_c2p_ERR_Loewner} shows that $K_{\rm pd}$ achieves a good $H_\infty$-approximation error. Figure \ref{fig_c2p_zpk_L} shows that the model $K_{\rm pd}$ preserves the imposed pole location. 

}
\subsection{CD player}
\label{sect_expl_CD}

\begin{figure}[!b]
\centering
\begin{subfigure}[b]{0.23\textwidth}
\includegraphics[width=\textwidth]{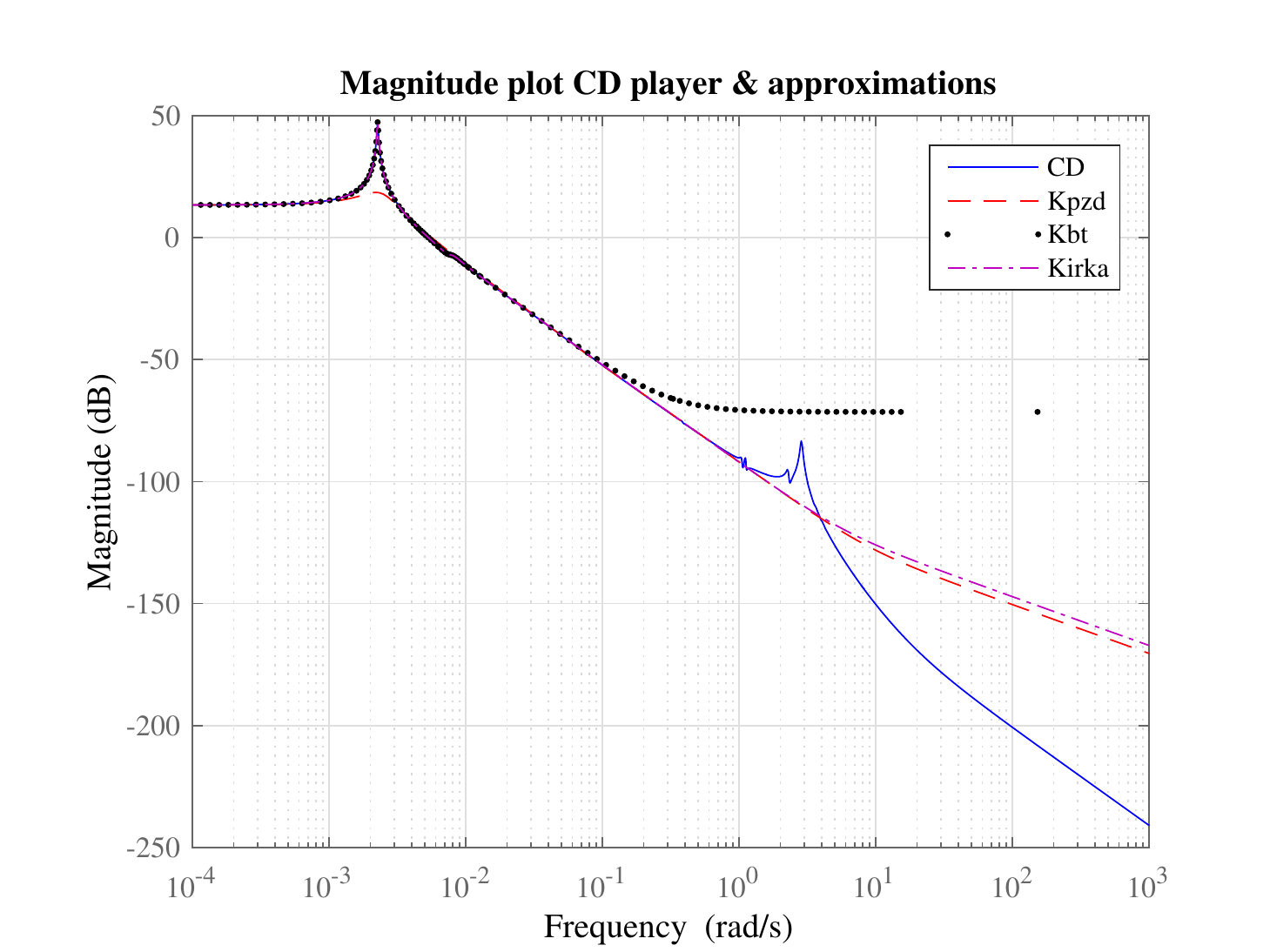}
\caption{$\nu=6,\ell=4,k=2$}\label{fig_642CDp_models}
\end{subfigure}
\hspace{0cm}
\begin{subfigure}[b]{0.23\textwidth}
\includegraphics[width=\textwidth]{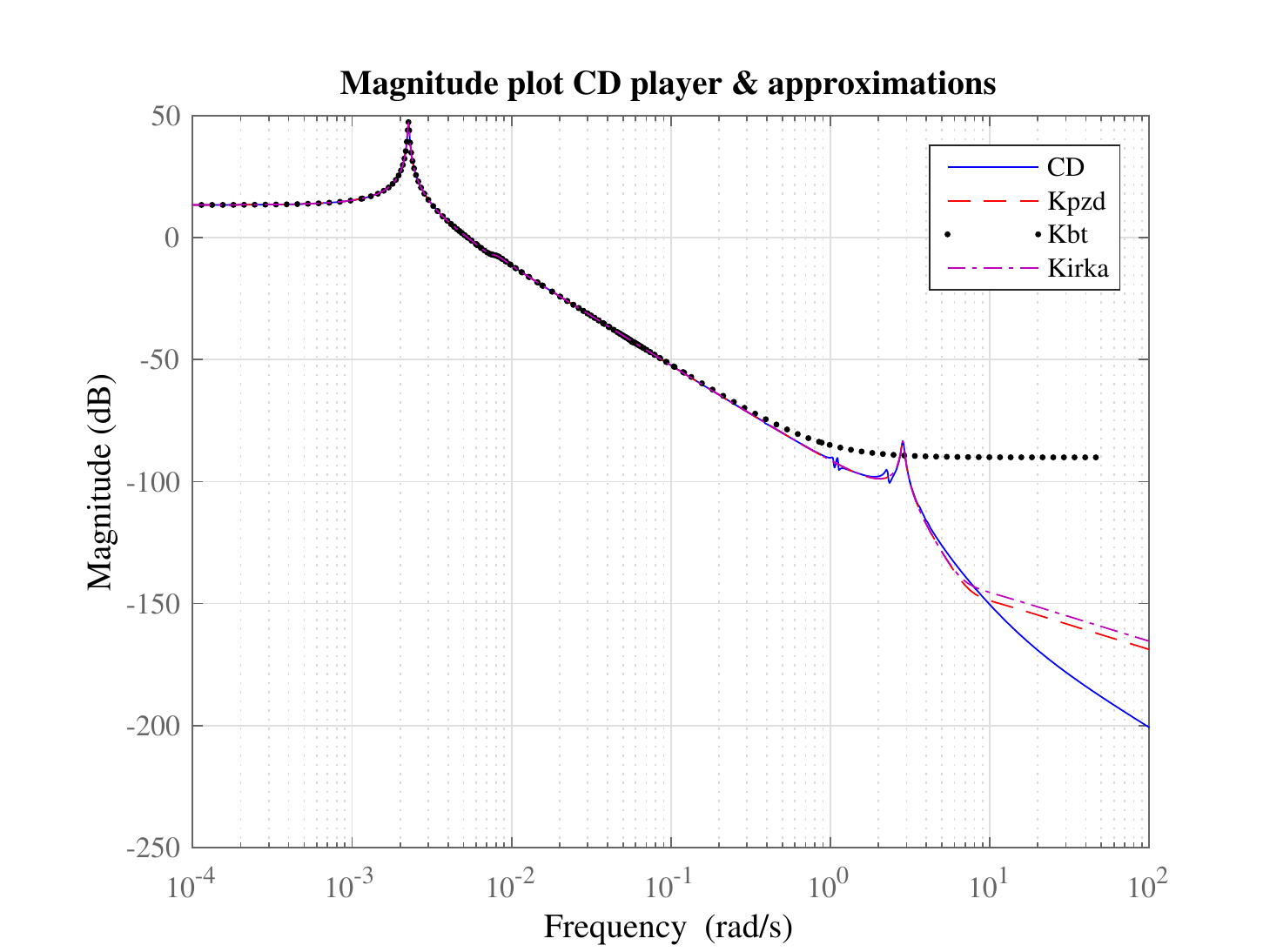}
\caption{$\nu=12,\ell=5,k=3$ }\label{fig_1253CDp_models}
\end{subfigure}
\caption{Magnitude plots of the models of the 120-th order CD player model (solid blue), the proposed models (dashed red), the $\nu$ order BT model (dotted black) and the $\nu$ order IRKA model (dash-dotted magenta)}\label{fig_642CDp}
\end{figure}

Consider the CD player, a single input single output LTI system with $n=120$, see, e.g.,  \cite{antoulas-2005,gugercin-antoulas-beattie-SIAM2008}.  Let 
$
S=\diag(s_1,\dots, s_\ell,s_{\ell+1},\dots, s_\nu)
$
such that $s_i, i=1:\nu$ is not an eigenvalue of $A$ and let $L=[1\ 1\ \dots\ \ 1]=[L_1\ L_2]$. Note that the matrix pair $(L,S)$ is observable. Furthermore, arbitrarily fix the sets of numbers
$
\{\lambda_1,\dots,\lambda_\ell\},
$
such that $s_j\ne\lambda_j, j=1:\ell$ and $\{z_1,\dots,z_k\}$, such that $\ell+k\leq\nu$. Let $\Pi$ be the solution of the Sylvester equation \eqref{eq_Sylvester_Pi}. Since $S$ is diagonal, $\Pi$  can be computed explicitly as in \cite[Lemma 2]{astolfi-TAC2010,i-astolfi-colaneri-SCL2014}.
We now write the family of $\nu$ order models $\Sigma_G$ as in \eqref{redmod_CPi}, parametrized in $G\in\mathbb R^{\nu}$, that match the moments of the CD player system at $\{s_1,\dots, s_\nu\}$.
Build the matrix $D_k=\diag(\theta_{k1},\dots,\theta_{k\nu}), \theta_{ki}=\lambda_k-s_i,\ i=1:\nu,\ k=1:\ell,$ and the numbers $\gamma_{ji}=z_j-s_i,\ i=1:\nu,\ j=1:k$. Also consider $\Ud$, the unique solution of the Sylvester equation $\Sd\Ud=\Ud A+RC$, where $R=L_2^T$ and $\Sd=\diag(s_1,\dots, s_\ell,s_{\ell+1},\dots, s_\nu)$. Note that since $\Sd$ is diagonal, $\Ud$  can be written explicitly as in \cite[Lemma 2]{astolfi-TAC2010,i-astolfi-colaneri-SCL2014}.
In the sequel we compute the matrices $G$ that yield the approximations $\Sigma_G$ (with the transfer function $K_G$) of order $\nu$ which
has $\ell$ poles at $\lambda_1,\dots,\lambda_\ell$, $k$ zeros at $z_1,\dots,z_k$ and satisfies the property that the derivatives of $K_G$ match the derivatives of $K$ at $s_{(\ell+k)+1},\dots,s_\nu$. We compute $G$ for $\nu=3, 6, 12$, for different values of $\ell$ and $k$. 
{\rev Note that, based on the results of Theorem \ref{thm_LisG} and Corllary \ref{rem_LisG}, instead of computing the family $\Sigma_G$, we can compute the Loewner matrices \eqref{eq_Loewner_general} and obtain identical results.%
}
We compare the results of the proposed method with the $\nu$ order balanced truncation approximation $K_{\text{BT}}$ and the $\nu$ order Iterative Rational Krylov Algorithm (IRKA) approximation, $K_\text{IRKA}$. The results of the simulations are presented in Table \ref{tab_comp}. Note that the set of interpolation points is chosen \emph{arbitrarily} in the complex plane and it contains zero for DC-gain preservation. Moreover, the selected interpolation points are also used for initializing the IRKA algorithm. Due to the lack of other constraints in the choice of the interpolation points, the approximation that matches $\nu$ derivatives of the given system at these points may yield unstable approximations. Furthermore, matching a significant number of derivatives numerically/practically ensures the decrease in the $H_2/H_\infty$-norm of the approximation error. 
{\rev
Figure \ref{fig_642CDp} illustrates the matching of the low frequency beahaviour performed by the proposed models. 
}
The example illustrates that the proposed approach yields reduced order models that allow for a trade-off between the good $H_2/H_\infty$-norm of the approximation error and the desired pole-zero placement. 
%


\section{Conclusions}


In this paper we have computed a low order approximation that matches the moments of a given large LTI system, has certain poles and zeros fixed and matches a number of selected high order moments. We have proposed linear algebraic systems whose solutions yield the models that meet the constraints. We have also extended the results to  data-based model reduction using  Loewner matrices that include the constraints. For future work, we extend the given results to other classes of systems such as infinite-dimensional and nonlinear.




\end{document}